\documentclass[11pt,a4paper]{article}

\usepackage{a4wide}
\usepackage{amsmath,amsthm,amssymb,amsfonts}
\usepackage{hyperref}
\usepackage{enumerate}
\usepackage{enumitem}
\usepackage{bm}
\usepackage[usenames]{color}
\usepackage[normalem]{ulem}
\usepackage{algorithm}
\usepackage{algorithmic}
\usepackage{amsmath}

\usepackage{diagbox}
\usepackage{subcaption}
\usepackage{xcolor}

\usepackage{graphicx}
\graphicspath{{./figures/}}

\DeclareMathOperator*{\argmin}{arg\,min}
\newtheorem{theorem}{Theorem}[section]
\newtheorem{lemma}[theorem]{Lemma}
\newtheorem{proposition}[theorem]{Proposition}
\newtheorem{corollary}[theorem]{Corollary}

\theoremstyle{definition}

\newtheorem{remark}[theorem]{Remark}

\definecolor{darkgreen}{rgb}{0.0,0.5,0.0}
\definecolor{darkblue}{rgb}{0.0,0.0,0.3}
\definecolor{nicosred}{rgb}{0.65,0.1,0.1}
\definecolor{light-gray}{gray}{0.6}
\definecolor{really-light-gray}{gray}{0.8}

\newcommand{\e}{\ensuremath{\varepsilon}}

\newcommand{\R}{\mathbb{R}}

\newcommand{\db}[1]{{#1}}

\newcommand{\mbt}[1]{{#1}}

\usepackage{authblk}

\makeatletter
\def\@fnsymbol#1{\ensuremath{%
    \ifcase#1\or 1\or 2\or 3\or 4\or 5\or 6\or 7\or 8\or 9%
    \else\@ctrerr\fi}}
\makeatother

\title{Polynomial diagrams for microstructure modelling}

\author{David P. Bourne\thanks{\texttt{d.bourne@hw.ac.uk}}}
\affil{Maxwell Institute for Mathematical Sciences and Department of Mathematics,\\ Heriot-Watt University, Edinburgh, EH14 4AS, UK}

\author{Maciej Buze\thanks{\texttt{m.buze@lancaster.ac.uk}}}
\affil{MARS: Mathematics for AI in Real-world Systems, School of Mathematical Sciences,\\ Lancaster University, Lancaster, LA1 4YF, UK}

\author{Thomas Gallou\"et\thanks{\texttt{thomas.gallouet@inria.fr}}}
\affil{Universit\'e Paris-Saclay, CNRS, Inria, ParMA Team, \\Laboratoire de Math\'ematiques d'Orsay, 91405 Orsay, France}

\author{Quentin M\'erigot\thanks{\texttt{quentin.merigot@universite-paris-saclay.fr}}}
\affil{Universit\'e Paris-Saclay, CNRS, Inria,\\ Laboratoire de Math\'ematiques d'Orsay, 91405 Orsay, France}
\affil{DMA, \'Ecole normale sup\'erieure, Universit\'e PSL, CNRS, 75005 Paris, France}

\begin{document}
%\maketitle

{%
  \setlength{\parskip}{-2pt}%
  \setlength{\baselineskip}{0pt}%  adjust to taste
  \maketitle
}
%\vspace*{-5pt}
%\begin{center}
%\emph{Dedicated to Robin Knops.}
%\end{center}
%\vspace{5pt}

\begin{abstract}
	We formulate a framework of \emph{polynomial diagrams}, which are a generalisation of power diagrams (PDs) and anisotropic power diagrams (APDs) allowing for boundaries between cells to be algebraic curves of a prescribed degree. We show that they arise naturally from rephrasing PDs (APDs) as first-degree (second-degree) instances of \emph{linear parametrised minimisation diagrams}. We also develop an efficient GPU-accelerated framework for fitting polynomial diagrams to image data using Legendre polynomials and by maximising a regularised \emph{concave} objective function adapted from classical logistic regression literature. A largely self-contained analysis of the optimisation algorithm is also provided, including identification of scale and gauge invariances and the limiting objective function as the regularisation parameter vanishes. We apply the algorithm to fit polynomial diagrams to electron backscatter diffraction images of steel.
\end{abstract}
\vspace{5pt}

\section{Introduction}
The modelling of polycrystalline materials is a central challenge in materials science and is of high relevance to industrial applications. 
The macroscopic response of a metal is strongly mediated by its grain structure: grain boundaries obstruct slip and promote strain heterogeneity, while grain size, morphology, and crystallographic texture enter directly into classical strength and hardening laws and into modern microstructure-sensitive constitutive modelling \cite{hall1951deformation,mecking1981kinetics,roters2010overview}.

Experimental microstructure data is commonly obtained from electron backscatter diffraction (EBSD) measurements,
and complementary datasets arise as discretised outputs of grain-growth, crystal-plasticity and FEM simulations \cite{anderson1984computer,moelans2008introduction,roters2010overview}. In all cases the data is inherently discrete, given as pixels/voxels on a set $\Omega$ approximating an ambient domain $\overline\Omega$. It can then be post-processed, e.g., using the open-source software MTEX \cite{MTEX}, resulting in a grain map $G\,\colon\,\Omega \to [N]$ (where $[N] := \{1,\dots,N\}$), which is a partition of $\Omega$  into $N$ grains, where $G(x)=k$ if the point $x$ belongs to grain $k$.

For analysis, storage, and fast geometric queries it is advantageous to replace a pixel/voxel level segmentation by a compact parametric representation whose induced cells approximate the observed grains.
Classical choices include Voronoi diagrams \cite{Aurenhammer1991}, power diagrams \cite{FWZL04, LLLFP11,QR18,DepriesterKubler2019,BKRS20,KSSB20,BPR23}, and, more recently, anisotropic power diagrams  \cite{FWZL04,KuhnSteinhauser08,LLLFP11,QR18}, which remain a topic of active research \cite{altendorf20143d,ABGLP15,VBWBSJ16,SWKKCS18,TR18,PFWKS21,AFGK23,AFGK23B, jung2024analytical} 
--
see also \cite{krause2026generating} for recent work on extending these approaches to account for crystallographic orientation of grains and \cite{plateau2025search} on fast computation of Apollonius diagrams on a GPU.
These diagrams have geometrically interpretable parameters, but their expressive power is limited when experimental boundaries depart from the assumed algebraic class. 

In this 
work 
we formulate a common framework for such representations termed \emph{parametrised minimisation diagrams} (PMDs): given a set of functions $\{f_1,\dots,f_N\}$ defined on the ambient space $\overline\Omega$, each point $x \in \Omega$ is \emph{arg-min assigned}, that is, assigned to the index attaining the minimum of $\{f_i(x)\}_{i \in [N]}$. We focus on the parametrised case
\[
f_i(x)=h(\theta_i,x),\qquad h\,\colon\,\R^K \times \overline\Omega \to \R\,
\]
and on fitting $\bm\theta=(\theta_1,\dots,\theta_N) \in \R^{K \times N}$ so that the resulting arg-min assignment matches a given grain map $G$ as much as possible.

Directly maximising pixel/voxel accuracy is, in general, a 
non-smooth nonlinear 
programming problem in $\bm\theta$. This was solved using stochastic optimisation in \cite{SBDWKKS2016}.
For the case of APDs, it was recently shown that this is a high-dimensional linear programming problem \cite{AFGK23}.
To obtain a smooth, concave maximisation problem
we use the well-known $\e$-scale soft assignment based on a log-sum-exp normalisation, leading to the standard multinomial logistic regression problem
\cite{BoydVandenberghe2004,HastieTibshiraniFriedman2009,Murphy2012}. See \cite{PFWKS21} for a recent use of similar ideas in the context of microstructure modelling, where the problem is formulated as a non-convex logistic regression problem (see Remark \ref{rem: related work}). For a recent survey of algorithms for fitting (anisotropic) power diagrams to image data, see \cite{Alpers2025}. The problem of fitting diagrams to \emph{incomplete} data, where the full grain map $G$ is not available and only the volumes and centroids of the grains are known, is studied for example in \cite{PetrichEtAl2019,QR18,bourne2024inverting,BourneGallouetMerigotNatale}. 

A foundational observation for the present work is that for the \emph{linear} class
\[
h_{\theta_i}(x)=\theta_i\cdot\eta(x),\qquad\text{ where }\;\eta\,\colon\,\Omega\to \R^K \;\text{ is some \emph{design function}},
\] 
the resulting objective function is strictly concave after removing a natural gauge invariance (common shifts of all $\theta_i$).
This yields a numerically robust optimisation problem even at large scale, and it allows us to analyse the behaviour of optimisers and near-optimisers, including the $\varepsilon$-scaling, the sharp $\varepsilon\to0$ limit, and the classical non-attainment phenomenon under perfect reconstruction (typically referred to as \emph{linear separation} in the logistic regression literature \cite{albert1984existence,santner1986note}).

To increase expressivity while preserving concavity, we extend the design function $\eta$ beyond the standard low-degree cases. Power diagrams and anisotropic power diagrams are recovered as degree one and two \emph{polynomial diagrams}, and we propose a general polynomial diagram framework with the design function $\eta$ containing all polynomial terms up to a degree $d$. In particular, we advocate rescaling to $[-1,1]^2$ and using a Legendre product basis, which spans the same polynomial space but typically improves the conditioning of the \emph{design matrix} $\eta(\Omega) \in \R^{K \times |\Omega|}$.
We solve the resulting optimisation problem on a
GPU using our Python library \textsc{PyAPD} \cite{PyAPD_paper,PyAPD_github}. We report numerical experiments on (i) synthetic data generated by power/anisotropic power diagrams and (ii) EBSD-derived grain maps of steel, illustrating the practical accuracy--complexity trade-off as the polynomial degree increases.
\paragraph{Outline of the paper.}
In Section~\ref{sec:pmd} we introduce the general framework of parametrised minimisation diagrams, formulate the fitting problem, and explain why the linear parametrisation class yields a concave multinomial logistic objective; standard power diagrams and anisotropic power diagrams are recovered in Section~\ref{sec:rel-std-diagrams} as first- and second-degree instances of this framework. In Section~\ref{sec:higher-d-d} we define polynomial diagrams of arbitrary degree, first in a monomial basis and then in a Legendre basis, and discuss their relation to the classical cases; see in particular Section~\ref{sec:Legendre}. Section~\ref{sec:analysis-linear-param} contains a largely self-contained analysis of the optimisation problem for linear PMDs, including the gradient and Hessian formulas, strict concavity after gauge fixing, the role of the $\varepsilon$-parameter, the $\varepsilon\to 0$ limit, and the existence/non-existence of maximisers depending on whether perfect reconstruction is possible. In Section~\ref{sec:numerics} we present GPU-based numerical experiments, first on synthetic power and anisotropic power diagram data (Section~\ref{sec:recon-apd}) and then on EBSD grain maps of steel (Section~\ref{sec:fit-ebsd}), followed by a discussion of degree selection, initialisation, recovery of physically meaningful parameters, and the practical role of $\varepsilon$ in Section~\ref{sec:discussion}. Finally, Section~\ref{sec:conclusion} summarises the main findings and outlines several directions for future work.
\paragraph{Acknowledgements}
\db{DPB and MB would like
to thank the UK Engineering and Physical Sciences Research Council (EPSRC) for financial
support via the grant EP/V00204X/1.}
MB~is \db{currently} supported by Research England under the Expanding Excellence in England (E3) funding stream, which was awarded to MARS: Mathematics for AI in Real-world Systems in the School of Mathematical Sciences at Lancaster University.\\
%TG is supported by \dots\\
QM is supported by Agence nationale de la recherche, through the PEPR PDE-AI project (ANR-
23-PEIA-0004). 
%\\
The authors thank Tata Steel 
\db{Netherlands} for providing EBSD datasets and Shadeform for providing access to cloud GPU resources.

\section{Parametrised minimisation diagrams}\label{sec:pmd}
\subsection{Setting and notation}\label{sec:setting-notation}
We consider an ambient space $\overline \Omega \subseteq \mathbb{R}^d$. To simplify the presentation, we take $d=2$ and
\[
\overline \Omega := [-1,1]^2,
\]
although our results are not limited to two dimensions. We take a discretisation of $\overline \Omega$ given by a set of pixels
\[
\Omega := \Bigg\{ \left(-1,-1\right) + \frac{1}{M}\left(k_1 - \frac12,\, k_2 - \frac12\right) \,\Bigg|\, k_i \in [2M] \Bigg\},
\]
where $M \in \mathbb{N}_+$ is the pixel resolution parameter and $[2M] := \{1,\dots,2M\}$. To be precise, the points in $\Omega$ are the centres of pixels, but we refer to them as pixels for brevity. Overall we have $|\Omega| = (2M)^2 = 4M^2$ pixels. \mbt{More generally, one may replace $\Omega$ by any finite set of sample points in $\overline\Omega$, for instance an unstructured pixel list.}

A \emph{grain map} is an assignment map $G \colon \Omega \to [N]$\mbt{, where $N > 1$,} which induces a tessellation of $\Omega$ into $N$ clusters, each given by $G_i := G^{-1}(\{i\})$. We assume that $G$ is given.

\mbt{A \emph{minimisation diagram} generated by a function
$\bm{f} = (f_1,\dots,f_N)\colon \overline \Omega \to \mathbb{R}^N$
is a tessellation of $\overline \Omega$ into $N$ cells:
\begin{equation}\label{Li_def}
L_i := \Bigl\{ x \in \overline\Omega \;\Big|\;
f_i(x) < f_j(x)\ \forall j<i,
\quad
f_i(x) \le f_j(x)\ \forall j>i
\Bigr\}.
\end{equation}
That is, ties are broken by assigning $x$ to the smallest index attaining the minimum. In particular,
\[
\overline \Omega = \bigcup_i L_i,\quad L_i \cap L_j = \emptyset \quad \text{for } i\neq j.
\]
}
We restrict our attention to \emph{parametrised minimisation diagrams (PMDs)}, namely
\begin{equation}\label{f-param-class}
f_i(\cdot) = h(\theta_i,\cdot), \quad h\,\colon\, \R^K \times \overline \Omega \to \R, \qquad \text{ (short-hand notation}\; h(\theta_i,\cdot) \equiv h_{\theta_i}(\cdot)).
\end{equation}
Thus $\theta_i \in \R^K$ is a set of parameters defining $f_i$ and $\bm{\theta} = (\theta_1,\dots,\theta_N) \in \R^{K \times N}$ is the overall set of parameters determining the diagram. 

\begin{remark}[Bold font and subscript convention]
	We reserve bold font for highest level collections of objects such as the collection $\bm{f} = (f_1,\dots,f_N)$ of functions generating a minimisation diagram or a collection of parameters $\bm \theta = (\theta_1,\dots,\theta_N)$. Depending on context a generic constituent function will be referred to as either $f$ or $f_i$ and a generic constituent parameter vector as $\theta$ or $\theta_i$. Lower level ordering for vectors $\theta,\theta_i \in \R^K$ will be separated by a comma, leading to notation $\theta_{i,j} \in \R$ and $\theta_{,j} \in \R$. The same convention applies to pixels $\Omega = \{x_1,\dots, x_{|\Omega|}\}$, with $x_{i,j} \in [-1,1]$ and the notation for a generic point $x = (x_{,1},x_{,2}) \in \overline\Omega$, or standard parameters defining (anisotropic) power diagrams (see Section~\ref{sec:rel-std-diagrams}).
\end{remark}
The PMD-induced grain map is
\[
F_{\bm \theta}\colon \overline \Omega \to [N],\quad F_{\bm \theta}(x) := \argmin_{i \in [N]} h_{\theta_i}(x).
\]
If the minimiser is not unique, we select the smallest index \mbt{(c.f., \eqref{Li_def})} and such diagrams are referred to as \emph{ambiguous}. An important special case is a \emph{fully ambiguous} diagram, which arises when $h_{\theta_i}(x) = h_{\theta_j}(x)$ for all $i,j \in [N]$ and all $x \in \overline \Omega$.

The characteristic function of the $i$th cell $L_i$ is given by
\begin{equation}\label{p}
(p_{\bm \theta})_i\colon \overline\Omega \to \{0,1\}, \quad (p_{\bm \theta})_i(x) := \delta(i,F_{\bm \theta}(x)),
\end{equation}
where $\delta$ is the Kronecker delta, so $\delta(k,j) = 1$ if $k=j$ and zero otherwise. For $\varepsilon > 0$, its smeared-out counterpart is
\begin{equation}\label{p-e}
(p_{\bm \theta}^\varepsilon)_i\colon \overline\Omega \to (0,1),\quad (p_{\bm \theta}^\varepsilon)_i(x) := \frac{\exp\left(-\frac1\varepsilon h_{\theta_i}(x)\right)}
{S_{\bm{\theta}}^\varepsilon(x)},
\quad S_{\bm{\theta}}^\varepsilon(x) := \sum_{j \in [N]} \exp\left(-\frac1\varepsilon h_{\theta_j}(x)\right).
\end{equation}
We can think of $(p_{\bm \theta}^\varepsilon)_i(x)$ as being the probability that pixel $x$ belongs to grain $i$. If $\bm \theta$ is such that the resulting diagram is fully ambiguous, then $(p_{\bm \theta}^\varepsilon)_i(x) = \frac{1}{N}$ for all $i \in [N]$ and all $x \in \overline\Omega$.

In what follows, we will investigate several possible choices of parameterisation in \eqref{f-param-class} and discuss optimisation approaches so that the PMD-induced grain map $F_{\bm \theta}$ is as close to the given grain map $G$ as possible.

\subsection{Optimisation approaches}
For now assume that the parameterisation function $h$ in \eqref{f-param-class} is fixed. The most natural measure of the discrepancy between $G$ and $F_{\bm \theta}$ is to simply count the number of pixels being assigned to correct grains, which gives the overall accuracy as 
\begin{equation}
\label{eq: ACC_G}    
{\rm Acc}^G(\bm \theta) := \frac{1}{|\Omega|}\sum_{x \in \Omega} (p_{\bm \theta})_{G(x)}(x)
\end{equation}
where $(p_{\bm \theta})_i$ is the characteristic function defined in \eqref{p}. By construction, ${\rm Acc}^G(\bm \theta) \in [0,1]$, where $0$ corresponds to a \emph{complete mismatch} (no pixels assigned correctly) and $1$ to \emph{perfect reconstruction} (all pixels assigned correctly). This leads to the optimisation problem
\begin{equation}\label{eq-max-match}
	\max_{\bm\theta \in \R^{K\times N}} {\rm Acc}^G(\bm \theta).
\end{equation}

The mapping $\bm \theta \mapsto {\rm Acc}^G(\bm \theta)$ is typically discontinuous, which renders solving \eqref{eq-max-match} directly tricky. One way to address this is to replace $p_{\bm \theta}$ with its smeared-out counterpart $p_{\bm \theta}^{\e}$ defined in \eqref{p-e}, leading to
\begin{equation}\label{eq-max-match-e}
\max_{\bm\theta \in \R^{K\times N}}{\rm Acc}_\e^G(\bm \theta), \quad {\rm Acc}_\e^G(\bm \theta) := \frac{1}{|\Omega|}\sum_{x \in \Omega} (p_{\bm \theta}^\e)_{G(x)}(x).
\end{equation}
\db{The exponential function in the definition of $p_{\bm \theta}^{\e}$ can lead to numerical overflow in finite-precision arithmetic when $\e$ is small. A}
%Reflecting the usage of the exponential function to define $p_{\bm \theta}^{\e}$, a 
numerically stable \mbt{counterpart of \eqref{eq-max-match-e} is the} optimisation problem 
\begin{equation}\label{log-res}
	\max_{\bm\theta \in \R^{K\times N}}\Phi^G_{\e}(\bm \theta), \quad \Phi^G_\e(\bm \theta):= \frac{1}{|\Omega|}\sum_{x \in \Omega} \log\left((p_{\bm\theta}^\e)_{G(x)}(x)\right).
\end{equation}
We recognise that this is the well-known multinomial logistic regression approach to clustering, which is 
standard in the statistics and machine-learning literature \cite{BoydVandenberghe2004,HastieTibshiraniFriedman2009,Murphy2012}.

The problem in \eqref{log-res} is particularly nice when the objective function $\Phi^G_\e$ is concave, as then every local maximiser is also a global maximiser. This will be analysed in detail in Section~\ref{sec:analysis-linear-param}. But first, in the next section, we will establish 
a broad class of PMDs for which the concavity of the objective function $\Phi^G_\e$ in \eqref{log-res} can be guaranteed.

\subsection{Linear parameterisation ensures concavity}
Recall \eqref{f-param-class}, fix some $K \in \mathbb{N}_+$ 
and suppose that $h\,\colon\,\R^K \times \overline\Omega \to \R$ is an arbitrary parameterisation giving rise to functions $\{h_{\theta_i}\}_{i \in [N]}$. We first rewrite the objective function $\Phi_{\varepsilon}^G$ from \eqref{log-res} as
\begin{equation}\label{log-res-2}
\Phi_{\varepsilon}^G(\bm \theta) = \frac{1}{|\Omega|}\sum_{x \in \Omega}\left(-\frac{1}{\e}h_{\theta_{G(x)}}(x) - \log S_{\bm \theta}^\e(x)\right)
\end{equation}
and we recall from \eqref{p-e} that
\[
S_{\bm{\theta}}^\e(x) = \sum_{j \in [N]} \exp\left(-\frac1\e h_{\theta_j}(x)\right).
\]
It is a standard assertion relying on Jensen's inequality \cite{BoydVandenberghe2004,blanchard2021accurately} that, for fixed $\e > 0$ and $x \in \Omega$, the function $\bm\theta \mapsto -\log S_{\bm \theta}^\e(x)$ is concave if $\theta \mapsto -h(\theta,x)$ is convex. But in this case
\begin{equation}\label{diff-of-convex}
\left(-\frac{1}{\e}h_{\theta_{G(x)}}(x) - \log S_{\bm \theta}^\e(x)\right)
\end{equation}
is a sum of a convex function and a concave function. A sufficient condition to ensure concavity is thus that the mapping $\theta \mapsto h(\theta,x)$ is linear, so of the form 
\begin{equation}\label{f-param-linear}
h(\theta,x) = \theta \cdot \eta(x)
\end{equation}
where $\eta\,\colon\, \overline \Omega \to \R^K$
is a \emph{design function}, giving rise to the well-known notion of a \emph{design matrix}:
\[
\eta(\Omega) \in \R^{K \times |\Omega|},\qquad (\eta(\Omega))_{ij} := \eta_i(x_j), \quad i \in [K],\; j \in [|\Omega|].
\]
This is exactly the well-known ``feature map'' approach in statistical learning: $\eta(x)$ is the feature (or basis) vector associated with input $x$, and $\eta(\Omega)$ is the corresponding \emph{design matrix} used in linear and generalised linear models; see, e.g., \cite{HastieTibshiraniFriedman2009,Murphy2012}. 
Equivalently, $\eta$ can be viewed as an explicit embedding of the input space into a (possibly higher-dimensional) feature space, in which the model is linear.

The fact that a linear parametrisation does indeed ensure concavity is well-known and will be discussed in Proposition~\ref{prop:concavity-grad-hess}.

\subsubsection{Relation to standard diagrams}\label{sec:rel-std-diagrams}
We will now show that two standard minimisation diagrams can be recovered as special cases of the linear parametrisation class \eqref{f-param-linear} for specific choices of $\eta$. 

\paragraph{Power diagrams.} 
The $i$th cell of 
%a 
\db{the}
power diagram generated by \db{the} $N$ seed points 
%\linebreak 
${\bm{y} = \{y_1,\dots,y_N\}}$, $y_i \in \R^2$, and weights $\bm w = \{w_1,\dots,w_N\}$, $w_i \in \R$, is given by
\[
L_i = \{ x \in \overline \Omega\,\mid\, \|x-y_i\|^2 - w_i \leq \|x-y_j\|^2 - w_j,\;\forall j \in [N]\}
.
\]
After simplifying, this can be equivalently rewritten as a PMD with $K=3$ and
\[
h(\theta,x) = \theta \cdot \eta(x), \qquad \eta(x) = (x_{,1},x_{,2},1).
\]
We also have a relation
\[
\theta = (-2y_{,1},-2y_{,2},y_{,1}^2 + y_{,2}^2 - w)
\]
and conversely 
\begin{equation}
\label{pd-lr-to-phys}    
y_{,1} = -\frac12\theta_{,1}, \quad y_{,2} = -\frac12\theta_{,2},\quad w = \frac14({\theta_{,1}})^2 + \frac14(\theta_{,2})^2 -\theta_{,3}.
\end{equation}

\paragraph{Anisotropic power diagrams.} The $i$th cell of 
%an 
\db{the} anisotropic power diagram generated by \db{the} $N$ seed points ${\bm{y} = \{y_1,\dots,y_N\}}$, $y_i \in \R^2$, weights $\bm w = \{w_1,\dots,w_N\}$, $w_i \in \R$, and anisotropy matrices $\bm{\mathrm{A}} = (\mathrm{A}_1,\dots,\mathrm{A}_N)$, $\mathrm{A}_i$ a 2x2 symmetric positive definite matrix, is given by
\begin{equation}\label{L-APD}
L_i = \{ x \in \overline \Omega\,\mid\, \|x-y_i\|_{\mathrm{A}_i}^2 - w_i \leq \|x-y_j\|^2_{\mathrm{A}_j} - w_j,\;\forall j \in [N]\}, \qquad \|z\|_A^2:=z \cdot A z.
\end{equation}
Expanding $\|x-y\|_{\mathrm A}^2-w = x \cdot {\mathrm A} x - 2({\mathrm A}y)\cdot x + (y \cdot {\mathrm A}y-w)$ shows this is a PMD with $K=6$ and
\begin{equation}\label{h-APD}
h(\theta,x)=\theta\cdot\eta(x),\qquad \eta(x)=(x_{,1}^2,x_{,1}x_{,2},x_{,2}^2,x_{,1},x_{,2},1).
\end{equation}
We also have a relation
\[
\theta=\Big(\mathrm{A}_{11},\ 2\mathrm{A}_{12},\ \mathrm{A}_{22},\ -2(\mathrm{A}_{11}y_{,1}+\mathrm{A}_{12}y_{,2}),\ -2(\mathrm{A}_{12}y_{,1}+\mathrm{A}_{22}y_{,2}),\ y \cdot \mathrm{A}y-w\Big),
\]
and conversely
\begin{equation}\label{apd-lr-to-phys}
\mathrm{A}=\begin{pmatrix}\theta_{,1}&\tfrac12\theta_{,2}\\[2pt]\tfrac12\theta_{,2}&\theta_{,3}\end{pmatrix},\quad
y=-\frac12\,\mathrm{A}^{-1}\binom{\theta_{,4}}{\theta_{,5}},\quad
w=y\cdot \mathrm{A} y - \theta_{,6} = \frac14\binom{\theta_{,4}}{\theta_{,5}}\cdot  \mathrm{A}^{-1}\binom{\theta_{,4}}{\theta_{,5}}-\theta_{,6}.
\end{equation}
For an arbitrary $\bm \theta \in \mathbb{R}^{6 \times N}$, the corresponding anisotropy matrices $\mathrm{A}_i$ may not be positive definite. It turns out, however, that the constraint of positive definiteness can be enforced while solving the \emph{unconstrained} optimisation problem $\max_{\bm \theta \in \mathbb{R}^{6 \times N}} \Phi^G_{\e}(\bm \theta)$; see Section \ref{sec: PD constraint}.

\begin{remark}[Prior related work]
\label{rem: related work}
    A similar discussion on transforming the standard APD description in terms of $(\bm w, \bm y, \bm{\mathrm{A}})$ from \eqref{L-APD} into a linear one based on \eqref{h-APD} and the substitution \eqref{apd-lr-to-phys} was also considered in \cite{AFGK23} in the context of formulating a linear programming problem to solve the unregularised problem \eqref{eq-max-match}. In \cite{PFWKS21} on the other hand, the authors do work with the smeared-out characteristic function $p^\e_{\bm \theta}$ defined in \eqref{p-e}, but using the standard APD description in terms of $(\bm w, \bm y, \bm{\mathrm{A}})$. Furthermore, instead of $\Phi_\e^G$, a binary cross-entropy objective function is considered, which in our linear $\bm \theta$ description
    amounts to solving
	\begin{equation}\label{eq-max-phi-bce}
		\max_{\bm\theta \in \R^{K\times N}}\,\Phi^{\rm BCE}_{\varepsilon}(\bm\theta), \quad \Phi^{\rm BCE}_{\varepsilon}(\bm \theta) := \frac{1}{|\Omega|}\sum_{i=1}^N\left(\sum_{x \in G_i}\log((p_{\bm \theta}^{\e})_{i}(x)) + \sum_{x \not\in G_i} \log(1-(p_{\bm \theta}^{\e})_{i}(x))\right).
	\end{equation}
	This is a well-known ``one-vs-all" approach to multi-class classification problems \cite{RifkinKlautau2004}, which is routinely used as an alternative to the multinomial logistic regression approach. While solving \eqref{eq-max-phi-bce} has its advantages, with gradient information arguably richer, the big disadvantage is that the loss of concavity is unavoidable. This follows from  
	\[
	\log(1-(p_{\bm \theta}^{\e})_{i}(x)) = \log\left(\sum_{\substack{j \in [N]\\ j \neq i}} \exp\left(-\frac1\e h_{\theta_j}(x)\right)\right) - \log\left(\sum_{j \in [N]} \exp\left(-\frac1\e h_{\theta_j}(x)\right)\right),
	\]
	which, similarly to the discussion 
    \db{following}
    %near 
    \eqref{diff-of-convex}, is a difference of two convex functions.  Since the underlying problem is one of exclusive multi-class classification (each pixel belongs to exactly one grain), a ``one-vs-all" binary approach should probably be avoided, as \mbt{(i)~it introduces no additional information (membership in one class automatically determines non-membership in all others),}
    (ii)~the computational cost is much increased (the complement of the set $G_i$ contains most pixels), and (iii)~the concavity is lost. The computational cost issue can be partially mitigated by employing batched stochastic optimisation techniques, as is done in \cite{PFWKS21}, but this of course makes the gradient information less rich. 
\end{remark}

\begin{remark}[Voronoi, multiplicative Voronoi, and Apollonius diagrams are not linear]
	Voronoi diagrams are the special case of power diagrams 
    where $w_i = 0$ for all $i$. In this case, we observe that while we can still write a Voronoi diagram as a PMD with 
	\[
	h(\theta,x) = \theta \cdot \eta (x), \quad \eta(x)=(x_{,1},x_{,2},1),\qquad
	\theta=(-2y_{,1},-2y_{,2},\|y\|^2),
	\]
	Voronoi diagrams are \emph{not} linear because the parameter space is constrained by
	\[
	\theta_{,3}-\tfrac14(\theta_{,1}^2+\theta_{,2}^2)=0.
	\]
	The same is true for multiplicatively weighted Voronoi diagrams, with cell formula
	\[
	L_i = \{ x \in \overline \Omega\,\mid\, w_i\|x-y_i\|^2 \leq w_j\|x-y_j\|^2,\;\forall j \in [N]\},
	\] 
	leading to an incomplete degree two monomial basis (a complete treatment of such polynomial diagrams will be given Section~\ref{sec:higher-d-d}),
	\[
	\tilde \eta^{(2)}(x)=(x_{,1}^2,\ x_{,2}^2,\ x_{,1},\ x_{,2},\ 1),\quad
	\theta=(w,\ w,\ -2wy_{,1},\ -2wy_{,2},\ w\|y\|^2),
	\]
	and with a nonlinear constraint
	\[
	\theta_{,3}^2 + \theta_{,4}^2 = 4\theta_{,1}\theta_{,5}.
	\]
    Note that one typically also has the constraint $\theta_{i,1} = \theta_{i,2} = w_i >0$ for all $i$.
    Similarly, an Apollonius diagram, with cell formula 
	\[
	L_i = \{ x \in \overline \Omega\,\mid\, \|x-y_i\| - w_i \leq \|x-y_j\| - w_j,\;\forall j \in [N]\},
	\] 
	cannot be expressed as a linear parametrisation since the seed $y_i$ enters inside the norm and squaring both sides of the inequality does not lead to a linear parametrisation either. 
\end{remark}

\section{Polynomial diagrams}\label{sec:higher-d-d}
The linearisation of standard diagrams presented in Section~\ref{sec:rel-std-diagrams} shows that a power diagram is a \emph{first degree} diagram, since its design function includes all zeroth- and first-order polynomial terms in $(x_{,1},x_{,2})$. Likewise, an anisotropic power diagram is a \emph{second degree} diagram. In this section we propose a natural generalisation to a class of \emph{polynomial diagrams} of an arbitrary degree $d > 0$. Note that we avoid the term \emph{higher-order diagram}, as this is a well-known term for a different type of generalisation, such as higher-order Voronoi diagrams \cite{Lee1982,Aurenhammer1991} and, e.g., their more recent extension to line segments \cite{papadopoulou2016higher}.

\subsection{Multi-index notation}
For $\alpha=(\alpha_1,\alpha_2)\in\mathbb{N}_0^2$ and $x=(x_{,1},x_{,2})\in\overline\Omega$, write
\[
|\alpha|:=\alpha_1+\alpha_2,\qquad x^\alpha:=x_{,1}^{\alpha_1}x_{,2}^{\alpha_2}.
\]
Fix $d\in\mathbb{N}_+$
and define the set of multi-indices up to total degree $d$ by
\[
\mathcal{A}_d:=\{\alpha\in\mathbb{N}_0^2:\ |\alpha|\le d\}, 
\qquad 
K_d:=|\mathcal{A}_d|=\frac{(d+1)(d+2)}{2}.
\]
We fix once and for all an ordering of $\mathcal{A}_d$ (e.g.\ graded lexicographic), and implicitly identify vectors in $\mathbb{R}^{K_d}$ with multi-indices $\alpha\in\mathcal{A}_d$.

\subsection{Monomial basis}
Define the degree $d$ monomial design function
\[
\eta^{(d)}:\overline\Omega\to\mathbb{R}^{K_d},\qquad 
\eta^{(d)}(x):=\big(x^\alpha\big)_{\alpha\in\mathcal{A}_d}.
\]
A \emph{monomial degree $d$ minimisation diagram} is a polynomial diagram of the linear form \eqref{f-param-linear} with $\eta=\eta^{(d)}$, i.e.,
\begin{equation}\label{eq:h-degree-d-monomial}
	h(\theta,x)=\theta\cdot \eta^{(d)}(x)
	=\sum_{\alpha\in\mathcal{A}_d}\theta_{,\alpha}\,x^\alpha,
	\qquad \theta\in\mathbb{R}^{K_d}.
\end{equation}
For $\bm\theta=(\theta_1,\dots,\theta_N)\in\mathbb{R}^{K_d\times N}$, the induced grain map is
\[
F_{\bm\theta}(x)=
\argmin_{i\in[N]}h_{\theta_i}(x),
\]
with the same tie-breaking convention as in Section~\ref{sec:setting-notation}, and the corresponding cells are \mbt{
\begin{equation}\label{eq:Li}
L_i
:=
\Bigl\{
x\in\overline\Omega \;\Big|\;
h_{\theta_i}(x) < h_{\theta_j}(x)\ \forall j<i,
\quad
h_{\theta_i}(x) \le h_{\theta_j}(x)\ \forall j>i
\Bigr\}.
\end{equation}}
The associated design matrix is
\[
\eta^{(d)}(\Omega)\in\mathbb{R}^{K_d\times|\Omega|},
\qquad
(\eta^{(d)}(\Omega))_{\alpha,j}:=x_j^\alpha,
\quad \alpha\in\mathcal{A}_d,\ j\in[|\Omega|].
\]
A monomial basis approach is presented primarily as the most intuitive extension of first-degree (power) and second-degree (anisotropic power) diagrams. For practical computations, we instead work with Legendre polynomials, which we will discuss now.
\subsection{Legendre basis}\label{sec:Legendre}
Since $\overline\Omega=[-1,1]^2$ (and any other rectangular domain can be scaled to $[-1,1]^2$), it is natural to replace monomials by (products of univariate) Legendre polynomials \cite[Chapter~17]{trefethen2019approximation}. Let $\{P_m\}_{m\ge 0}$ denote the standard Legendre polynomials on $[-1,1]$, normalised by $P_0\equiv 1$, $P_1(t)=t$ and satisfying the three-term recurrence
\[
(m+1)P_{m+1}(t)=(2m+1)\,t\,P_m(t)-m\,P_{m-1}(t).
\]
Legendre polynomials are orthogonal on $[-1,1]$ with respect to uniform weight,
\[
\int_{-1}^1 P_m(t)\,P_n(t)\,dt=\frac{2}{2n+1}\,\delta_{mn}.
\]
This typically leads to the associated columns of the design matrix being less correlated, improving conditioning of the design matrix and stability of gradient-based optimisation, especially for larger $d$. 

Define the Legendre product indexed by $\alpha=(\alpha_1,\alpha_2)\in\mathcal{A}_d$ as
\[
\psi_\alpha(x):=P_{\alpha_1}(x_{,1})\,P_{\alpha_2}(x_{,2}).
\]
Then the degree $d$ Legendre design function is
\[
\eta_{\mathrm{L}}^{(d)}:\overline\Omega\to\mathbb{R}^{K_d},\qquad 
\eta_{\mathrm{L}}^{(d)}(x):=\big(\psi_\alpha(x)\big)_{\alpha\in\mathcal{A}_d}
=\big(P_{\alpha_1}(x_{,1})P_{\alpha_2}(x_{,2})\big)_{\alpha\in\mathcal{A}_d}.
\]
A \emph{Legendre degree $d$ minimisation diagram} is a polynomial diagram of the linear form \eqref{f-param-linear} with $\eta=\eta_{\mathrm{L}}^{(d)}$, i.e.,
\begin{equation}\label{eq:h-degree-d-legendre}
	h(\theta,x)=\theta\cdot \eta_{\mathrm{L}}^{(d)}(x)
	=\sum_{\alpha\in\mathcal{A}_d}\theta_{,\alpha}\,P_{\alpha_1}(x_{,1})P_{\alpha_2}(x_{,2}),
	\qquad \theta\in\mathbb{R}^{K_d}.
\end{equation}
The induced grain map $F_{\bm\theta}$ and cells $L_i$ are defined as before. The associated design matrix is
\[
\eta_{\mathrm{L}}^{(d)}(\Omega)\in\mathbb{R}^{K_d\times|\Omega|},
\qquad
(\eta_{\mathrm{L}}^{(d)}(\Omega))_{\alpha,j}:=P_{\alpha_1}(x_{j,1})P_{\alpha_2}(x_{j,2}),
\quad \alpha\in\mathcal{A}_d,\ j\in[|\Omega|].
\]

\begin{remark}[Consistency with standard cases]\label{rem:basis_change}
For $d=1$, the monomial design gives
\[
\eta^{(1)}(x)=(x_{,1},x_{,2},1),
\]
recovering the power-diagram linearisation.
For $d=2$, the monomial design gives
\[
\eta^{(2)}(x)=(x_{,1}^2,x_{,1}x_{,2},x_{,2}^2,x_{,1},x_{,2},1),
\]
recovering \eqref{h-APD}.
For the Legendre variant, note that $P_0\equiv 1$ and $P_1(t)=t$, hence $\eta_{\mathrm{L}}^{(1)}(x)$ spans the same linear feature space as the monomial case (up to re-ordering), while higher $d$ provide an alternative (often better conditioned) polynomial basis. Since both monomial and Legendre bases span the same polynomial space $\mathbb{P}_d(\mathbb{R}^2)$ (total degree $\le d$), there exists a unique and easily computable invertible matrix
$T^{(d)}\in\mathbb{R}^{K_d\times K_d}$ allowing us to conveniently convert between them. 
\end{remark}

\begin{remark}[Geometric interpretation of cell boundaries]\label{rem:geom-boundaries}
For any $i\neq j$, the interface between cells $L_i$ and $L_j$ is contained in
\[
\{x\in\overline\Omega:\ h_{\theta_i}(x)=h_{\theta_j}(x)\}
=
\{x\in\overline\Omega:\ (\theta_i-\theta_j)\cdot \eta^{(d)}(x)=0\},
\]
which is an algebraic curve of degree at most $d$ (independent of the chosen basis). 
\end{remark}

Extensive numerical tests showcasing the fitting of Legendre polynomial diagrams to grain maps will be presented in Section~\ref{sec:numerics}. Let us first, in the next section, provide a mathematical analysis of optimising abstract linear PMDs. In particular, the motivation for using Legendre polynomials is discussed in Remark~\ref{rem:conditioning-basis}.

\section{Analysis of optimising linear PMDs}\label{sec:analysis-linear-param}
Throughout this section we assume the linear parametrisation
\begin{equation}\label{eq:linear-h}
	h_{\theta_i}(x)=\theta_i\cdot \eta(x),\qquad x\in \overline \Omega,\ i\in[N],
\end{equation}
with design function $\eta:\overline \Omega \to\mathbb{R}^K$.
For $\varepsilon>0$ define the soft assignment
\[
(p_{\bm\theta}^\varepsilon)_i(x):=\frac{\exp\!\left(-\frac1\varepsilon h_{\theta_i}(x)\right)}
{\sum_{j=1}^N\exp\!\left(-\frac1\varepsilon h_{\theta_j}(x)\right)}
\]
and recall that the objective function is 
\begin{equation}\label{eq:Phi}
	\Phi_\varepsilon^G(\bm\theta):=\frac{1}{|\Omega|}\sum_{x\in\Omega}\log\big((p_{\bm\theta}^\varepsilon)_{G(x)}(x)\big)\le 0.
\end{equation}
We also recall the induced hard assignment map
\[
F_{\bm\theta}(x):=\argmin_{i\in[N]}h_{\theta_i}(x),
\]
with the same tie-breaking rule as in Section~\ref{sec:setting-notation}.
Finally, define the fraction of misassigned pixels
\[
{\rm Err}^G(\bm \theta) :=\frac{1}{|\Omega|}\Big|\{x\in\Omega:\ F_{\bm\theta}(x)\neq G(x)\}\Big|.
\]
Note that by construction ${\rm Err}^G(\bm \theta) = 1 - {\rm Acc}^G(\bm \theta)$ (defined in \eqref{eq: ACC_G}).

We now present a short, self-contained analysis of the 
optimisation problem $\max \Phi_\varepsilon^G$,
which is equivalent to softmax (multinomial logistic) regression with log-sum-exp normalisation. We do not claim novelty, as these results are standard in the statistics and machine-learning literature \cite{BoydVandenberghe2004,HastieTibshiraniFriedman2009,Murphy2012}, including the classical discussion of (non-)existence of maximisers under separation, which to the best of our knowledge can be traced back to \cite{albert1984existence,santner1986note}.

We begin by showing that since the parametrisation is linear, the parameter $\e >0$ can always be absorbed into $\theta$.

\begin{lemma}[$\varepsilon$--parameter scaling]\label{lem:eps-scaling}
	For any $\varepsilon>0$ and $\bm\theta\in\mathbb{R}^{K\times N}$, define $\bm\vartheta:=\bm\theta/\varepsilon$
	(i.e., %\ 
    $\vartheta_i=\theta_i/\varepsilon$ for all $i$). Then, for all $x \in \Omega$,
	\[
	F_{\bm\vartheta}(x)=F_{\bm\theta}(x), \qquad (p_{\bm\theta}^{\varepsilon})_i(x)=(p_{\bm\vartheta}^{1})_i(x)\qquad i\in[N]
	\]
	and in particular $\Phi_\varepsilon^G(\bm\theta)=\Phi_1^G(\bm\vartheta)$.
\end{lemma}
\begin{proof}
	This follows by a direct computation.
\end{proof}

%db{[DB: Is it also worth noting that $F_{\bm\vartheta}(x)=F_{\bm\theta}(x)$, i.e., the cells $L_i$ are unchanged by scaling?]}

Next we turn to the so-called gauge invariance and fixing. 
\begin{lemma}[Gauge invariance]\label{lem:gauge}
	For any $c\in\mathbb{R}^K$, define $\bm\theta'=(\theta'_1,\dots,\theta'_N)$ by $\theta'_i:=\theta_i+c$ for all $i$.
	Then for all $x\in\Omega$,
	\[
	F_{\bm\theta'}(x)=F_{\bm\theta}(x),\qquad (p_{\bm\theta'}^\varepsilon)_i(x)=(p_{\bm\theta}^\varepsilon)_i(x),\qquad i\in[N],
	\]
	and hence $\Phi_\varepsilon^G(\bm\theta')=\Phi_\varepsilon^G(\bm\theta)$.
\end{lemma}
\begin{proof}
	This follows by a direct computation.
\end{proof}
\begin{remark}[Gauge fixing]\label{rem:gauge-fix}
	We can factor out the gauge invariance by $K$ independent linear constraints removing the common-shift direction. The simplest one  is to treat one of the cells as the reference one and, e.g., fix $\theta_N\equiv 0$.
\end{remark}

\begin{proposition}[Concavity of $\Phi^G_{\varepsilon}$ and gradient and Hessian formulas]\label{prop:concavity-grad-hess}
	Fix $\varepsilon>0$. Under \eqref{eq:linear-h}, the objective function $\Phi_\varepsilon^G$ in \eqref{eq:Phi} is concave. Moreover, for each $i\in[N]$,
	\begin{equation}\label{eq:grad}
		\nabla_{\theta_i}\Phi_\varepsilon^G(\bm\theta)
		=
		-\frac{1}{\varepsilon|\Omega|}\sum_{x\in\Omega}\Big(
        \delta_{iG(x)}
        -(p_{\bm\theta}^{\varepsilon})_i(x)\Big)\,\eta(x).
	\end{equation}
    Similarly, for $i,k\in[N]$, the $(i,k)$-block of the Hessian is
\begin{equation}\label{eq:hess}
		\nabla_{\theta_k}\nabla_{\theta_i}\Phi_\varepsilon^G(\bm\theta)
        = -\frac{1}{\varepsilon^2|\Omega|}\sum_{x\in\Omega} (p_{\bm\theta}^{\varepsilon})_i(x)\big(\delta_{ik}-(p_{\bm\theta}^{\varepsilon})_k(x)\big)\,\bigl(\eta(x) \otimes \eta(x)\bigr),
	\end{equation}
	and the Hessian is negative semidefinite (with a null direction corresponding to the gauge shift in Lemma~\ref{lem:gauge}).
\end{proposition}

\begin{proof}
	Formulas \eqref{eq:grad} and \eqref{eq:hess} follow from a direct computation. The negative semidefiniteness of the Hessian follows from the proof of Proposition~\ref{prop:strict-concavity} below.
\end{proof}

\begin{proposition}[Strict concavity of $\Phi^G_{\varepsilon}$]
\label{prop:strict-concavity}
    Assume that the set $\{\eta(x) : x \in \Omega\}$ spans $\mathbb{R}^K$. Then the nullspace of the Hessian consists only of vectors of the form $\mathbf{v} = (c,\dots,c)$ for some $c\in\mathbb{R}^K$, and the function $\Phi_\varepsilon^G$ is strictly concave on the orthogonal complement of this nullspace.
\end{proposition}

\begin{proof}
To improve readability, set $p_i(x) := (p_{\bm\theta}^{\varepsilon})_i(x)$, and for every $x\in\Omega$ denote by
\[
P_x = (p_1(x),\dots,p_N(x))
\]
the associated probability distribution. Let $\mathbf{v} = (v_1, \dots, v_N)$ be a block vector, where $v_i \in \mathbb{R}^K$. The quadratic form associated with the Hessian matrix of $\Phi_\varepsilon^G(\bm\theta)$ is
    \begin{align*}
        \mathbf{v}^\top \nabla^2 \Phi_\varepsilon^G(\bm\theta)\, \mathbf{v}
        &=-\frac{1}{\varepsilon^2|\Omega|} \sum_{x\in\Omega} \left[ \sum_{i=1}^N p_i(x)\, (v_i \cdot \eta(x))^2 - \left( \sum_{i=1}^N p_i(x)\, (v_i \cdot \eta(x)) \right)^2 \right]\\
        &=-\frac{1}{\varepsilon^2|\Omega|} \sum_{x\in\Omega} \mathrm{Var}_{P_x}\bigl(i\mapsto v_i\cdot\eta(x)\bigr),
    \end{align*}
    where $\mathrm{Var}_{P_x}(\cdot)$ denotes the variance with respect to the probability distribution $P_x$.
    Since the variance is non-negative, the quadratic form is non-positive, proving concavity of $\Phi_\varepsilon^G$.

    The formula above shows that $\mathbf{v}$ belongs to the nullspace of the Hessian of $\Phi_\varepsilon^G$ if and only if for all $x$, the variance of $i\mapsto v_i\cdot\eta(x)$ vanishes. Since the probabilities $p_i(x)>0$ are all strictly positive, this implies that $i \mapsto v_i\cdot\eta(x)$ is constant for all $x$. In particular, for each $i$ the vector $v_i - v_1$ is orthogonal to $\eta(x)$ for all $x \in \Omega$. Since by assumption the vectors $\eta(x)$, $x\in\Omega$, span the entire space $\mathbb{R}^K$, we get $v_i \equiv v_1$ for all $i$, which corresponds exactly to the gauge shift invariance.
\end{proof}
\mbt{
\begin{remark}[Conditioning and choice of basis]\label{rem:conditioning-basis}
    The Hessian in \eqref{eq:hess} is built from weighted second-moment matrices of the form
    \[
    \sum_{x\in\Omega} (p_{\bm\theta}^{\varepsilon})_i(x)\,\eta(x)\otimes\eta(x),
    \]
    and the quadratic-form representation in the proof of Proposition~\ref{prop:strict-concavity} shows that their conditioning directly influences the conditioning of the optimisation problem. This motivates choosing a basis $\eta$ for which these matrices are better conditioned and, in favourable cases, close to diagonal. In particular, after rescaling the domain to $[-1,1]^2$, tensor-product Legendre polynomials are a natural choice: in the continuum they are orthogonal with respect to the uniform measure \cite[Chapter~17]{trefethen2019approximation}, and on sufficiently regular discrete grids the corresponding discrete Gram matrices are typically much better conditioned, and often closer to diagonal, than those arising from the monomial basis \cite{gautschi2020stable}. Although the
    weights $(p_{\bm\theta}^{\varepsilon})_i(x)$ mean that the matrices above are not, in general,
    exactly diagonal in the Legendre basis, one still expects a substantial conditioning advantage
    over monomials whenever the induced discrete measures are not too far from uniform.  This is one of the motivations for the Legendre-based parametrisation discussed in Section~\ref{sec:Legendre}.
\end{remark}
}

Next we investigate what happens as $\varepsilon \to 0$.

\begin{proposition}[Uniform $\varepsilon\to0$ limit of the rescaled objective]\label{prop:uniform-eps0}
	Define $\mathcal{E}_\varepsilon(\bm\theta):=-\varepsilon\,\Phi_\varepsilon^G(\bm\theta)$.
	Then for every $\bm\theta\in\mathbb{R}^{K\times N}$ and every $\varepsilon>0$,
	\begin{equation}\label{eq:uniform-bound}
		0 \le \mathcal{E}_\varepsilon(\bm\theta)-\mathcal{E}_0(\bm\theta)\le \varepsilon\log N,
	\end{equation}
	where the limit functional is
	\begin{equation}\label{eq:E0}
		\mathcal{E}_0(\bm\theta)
		:=
		\frac{1}{|\Omega|}\sum_{x\in\Omega}\Big(h_{\theta_{G(x)}}(x)-\min_{j\in[N]}h_{\theta_j}(x)\Big)\ \ge 0.
	\end{equation}
	In particular, $\mathcal{E}_\varepsilon\to\mathcal{E}_0$ \emph{uniformly} on $\mathbb{R}^{K\times N}$ as $\varepsilon\to0$. Moreover, if perfect reconstruction of the given grain map $G$ is possible (i.e., if there exists $\bm\theta$ such that $F_{\bm\theta}(x)=G(x)$ for all $x\in\Omega$), then $\min \mathcal{E}_0=0$.
\end{proposition}

\begin{proof}
	For each $x\in\Omega$, write $a_j(x):=-h_{\theta_j}(x)$.
	Then
    \[
    \varepsilon\log\sum_{j}e^{a_j(x)/\varepsilon}
    \]
    is the standard log-sum-exp smooth approximation of $\max_j a_j(x)$, with the sharp bounds (see, e.g., \cite{blanchard2021accurately})
	\[
	\max_j a_j(x)\ \le\ \varepsilon\log\sum_{j}e^{a_j(x)/\varepsilon}\ \le\ \max_j a_j(x)+\varepsilon\log N.
	\]
    Substituting $a_j(x)=-h_{\theta_j}(x)$, rearranging, and averaging over $x$ yields
    \[
    0 \ \le \ \frac{1}{|\Omega|} \sum_{x \in \Omega} \varepsilon \log \sum_j \exp \left( -\frac{h_{\theta_j}(x)}{\varepsilon} \right) + \min_{j \in [N]} h_{\theta_j}(x) \ \le \ \varepsilon \log N.
    \]
    Then \eqref{eq:uniform-bound} follows immediately from the definitions of $\mathcal{E}_\varepsilon$ and $\mathcal{E}_0$. For perfect reconstruction, one has $h_{\theta_{G(x)}}(x) = \min_{j \in [N]} h_{\theta_j}(x)$ for all $x \in \Omega$, and hence $\min \mathcal{E}_0=0$, as claimed.
\end{proof}

The final set of results concerns the question of (non-)existence of maximisers of $\Phi_\varepsilon^G$, which is known to depend on whether the diagram representation can perfectly recover the given grain map $G$ \cite{albert1984existence,santner1986note}.

\begin{proposition}[Perfect reconstruction implies the supremum is not attained]\label{prop:sep-sup}
	Assume that $N>1$ and there exists $\bm\theta^0$ such that ${\rm Err}^G(\bm \theta^0) = 0$ (so $F_{\bm\theta^0}(x)=G(x)$ for all $x\in\Omega$) and $G(x)$ is the strict minimiser of $i \mapsto h_{\theta_i^0}(x)$ for all $x \in \Omega$. Then, for any fixed $\varepsilon>0$,
	\[
	\sup_{\bm\theta}\Phi_\varepsilon^G(\bm\theta)=0,
	\]
	and the supremum is not attained (i.e., there is no maximiser in $\mathbb{R}^{K\times N}$, even after gauge fixing).
\end{proposition}

\begin{proof}
	Let $\bm\theta^t:=t\,\bm\theta^0$.
	For each $x$, since $G(x)$ is a strict minimiser of $h_{\theta_i^0}(x)$ among $i\in[N]$, one has
    \[
    (p_{\bm\theta^t}^\varepsilon)_{G(x)}(x)
    = \frac{1}{\displaystyle 1 + \sum_{j \ne G(x)} \left( \frac{\exp(-h_{\theta_j^0}(x)/\varepsilon)}{\exp(-h_{\theta_{G(x)}^0}(x)/\varepsilon)}\right)^t } \to 1 \quad \text{as } t \to \infty,
    \]
    hence $\Phi_\varepsilon^G(\bm\theta^t)\to 0$.
	On the other hand, for any finite $\bm\theta$ and any $x$, one has $(p_{\bm\theta}^\varepsilon)_{G(x)}(x)<1$ because $N>1$, so $\Phi_\varepsilon^G(\bm\theta)<0$ and thus the supremum cannot be attained.
\end{proof}

\begin{lemma}[Pixel misassignment has a cost]\label{lem:log2-penalty}
	Assume that $N >1$. For any $\bm\theta$ and any $x\in\Omega$, if $F_{\bm\theta}(x)\neq G(x)$ then
	\[
	(p_{\bm\theta}^\varepsilon)_{G(x)}(x)\le \frac12
	\qquad\text{and hence}\qquad
	\log\big((p_{\bm\theta}^\varepsilon)_{G(x)}(x)\big)\le -\log 2.
	\]
	Consequently,
\begin{equation}\label{eq:Phi-vs-Err}
		\Phi_\varepsilon^G(\bm\theta)\le -(\log 2)\,{\rm Err}^G(\bm \theta).
	\end{equation}
\end{lemma}

\begin{proof}
If $F_{\bm\theta}(x)\neq G(x)$, then
\[
(p_{\bm\theta}^\varepsilon)_{G(x)}(x)
\le \frac{\exp(-h_{\theta_{G(x)}}(x)/\varepsilon)}{\exp(-h_{\theta_{G(x)}}(x)/\varepsilon) + \exp(-h_{\theta_{F_{\bm \theta}(x)}}(x)/\varepsilon)}
\le \frac 12
\]
as required.
\end{proof}

The following is a direct consequence.

\begin{corollary}[Near-optimality implies (near-)perfect reconstruction]\label{cor:near-opt}
	For any $\delta>0$,
	\[
	\Phi_\varepsilon^G(\bm\theta) > - \delta
    \quad\Longrightarrow\quad
	{\rm Err}^G(\bm\theta)<\frac{\delta}{\log 2}.
	\]
	In particular, if $\delta<\frac{\log 2}{|\Omega|}$, then
    \[
    \Phi_\varepsilon^G(\bm\theta) > - \delta
    \]
    implies ${\rm Err}^G(\bm\theta)=0$,
	i.e., $F_{\bm\theta}(x)=G(x)$ for all $x\in\Omega$.
\end{corollary}

Note that this in particular implies that if a perfect reconstruction is possible, then any approximately optimal $\bm\theta$ achieves it.

\begin{proposition}[Imperfect reconstruction case: existence of a maximiser]\label{prop:existence-nonseparable}
	If there is no $\bm\theta$ such that ${\rm Err}^G(\bm\theta) = 0$ (which is equivalent to $F_{\bm\theta}(x)=G(x)$ for all $x\in\Omega$), then $\Phi_\varepsilon^G$ admits a maximiser (which is not unique due to the gauge invariance). Assume in addition that the set $\{\eta(x) : x \in \Omega\}$ spans $\mathbb{R}^K$. Then $\Phi_\varepsilon^G$ has a unique maximiser on the subspace $\{ (c,\ldots,c) : c \in \mathbb{R}^K \}^\perp$. Consequently $\Phi_\varepsilon^G$ also has a unique maximiser on the gauge-fixed subspace $\{ \bm \theta \in \mathbb{R}^{K \times N}: \theta_N = 0\}$.
\end{proposition}

\begin{proof}
	See \cite{albert1984existence,santner1986note} and Proposition~\ref{prop:strict-concavity}.
\end{proof}

\section{Numerical experiments}\label{sec:numerics}

\subsection{Implementation and computational setting}
In all the experiments we optimise the multinomial logistic objective function $\Phi_\varepsilon^G$ defined in \eqref{eq:Phi}
under the linear parametrisation \eqref{eq:linear-h}, using Legendre polynomial diagrams (Section~\ref{sec:higher-d-d}) of degrees ${d\in\{1,\dots,7\}}$. We use the gauge fixing from Remark~\ref{rem:gauge-fix} by setting $\theta_N\equiv 0$ throughout.

The implementation is part of the GPU-oriented library \textsc{PyAPD} \cite{PyAPD_paper,PyAPD_github}.
The evaluation of $\Phi_\varepsilon^G$ and its gradient uses a \textsc{PyKeOps} backend \cite{CFGCD21}, avoiding explicit formation of the full \mbt{ cost matrix 
\[
\bigl(h_{\theta_i}(x)\bigr)_{i\in[N],\,x\in\Omega}\in\mathbb{R}^{N\times |\Omega|},
\]
whose $(i,x)$-entry is the cost of assigning pixel $x$ to grain $i$.}
All runs were performed on a single NVIDIA A100 GPU using double precision arithmetic.  As discussed in \cite{PyAPD_paper}, for a combination of $N$ (number of grains), $|\Omega|$ (number of pixels) and $d$ (degree of the diagram) large enough, our implementation ensures near 100\% utilisation of the GPU.

All data used and generated in the numerical experiments is available in a reproducible manner on GitHub through our library, PyAPD \cite{PyAPD_github}.

\subsection{Optimisation protocol}
To showcase the robustness of the method, we first initialise at the fully ambiguous state $\bm\theta^{(0)}=\bm 0$, which implies that $(p_{\bm\theta^{(0)}}^\varepsilon)_i\equiv 1/N$. A second initialisation choice comes from the heuristic guess from \cite{TR18} -- see Section~\ref{sec:heuristic_guess} for details. 

We maximise $\Phi_\varepsilon^G$ using the L-BFGS method \cite{LN89}. In all cases we rely on the automatic differentiation available in \textsc{PyTorch} to compute derivatives. This is known to be faster than using the explicit formulas 
from Proposition~\ref{prop:concavity-grad-hess} \cite{PyAPD_paper}. Reflecting the fact that $\Phi_{\varepsilon}^G$ becomes very flat near the optimum, we always run the optimisation for a fixed moderate number of iterations, either 1000 or 2000, and fix the regularisation parameter to be $\e = 10^{-2}$. A robust \emph{annealing schedule} \cite{van1987simulated} of systematically decreasing $\e$ and a bespoke stopping criterion will be considered in future work (see also the last paragraph of Section~\ref{sec:role-of-eps} for a related discussion).

\subsection{Datasets and common discretisation}
We always take $\overline\Omega=[-1,1]^2$ with discretisation $\Omega$, either as in Section~\ref{sec:setting-notation} or provided by the EBSD dataset. A grain map $G:\Omega\to[N]$ is provided either synthetically
(from a known power diagram or anisotropic power diagram), with discretisation parameter $M$ chosen as suggested in \cite[Table~1]{PyAPD_paper}, or from experimental EBSD data
(pre-processing in \textsc{MTEX} \cite{MTEX}).

\subsection{Reconstruction of (anisotropic) power diagrams}\label{sec:recon-apd}
To show that the method works as expected, we first consider synthetic grain maps $G$ generated by (A) a power diagram (PD) and (B)--(E) anisotropic power
diagrams (APD) with increasing anisotropy (elongation of the grains). These experiments are in the perfect reconstruction regime: there exists $\bm\theta^\star$ such that
$F_{\bm\theta^\star}(x)=G(x)$ for all $x\in\Omega$,
hence $\sup\Phi_\varepsilon^G=0$ is not attained (Proposition~\ref{prop:sep-sup}) and we look for a near optimum achieving perfect reconstruction (Corollary~\ref{cor:near-opt}). \mbt{In practice, to counter the increasing cost of L-BFGS iterates as $\|\bm \theta\|$ grows, we impose a fixed budget of $2000$ iterations and accept a (negligibly) small value ${\rm Err}^G$.} The results are presented in Figure~\ref{fig:1}.

\begin{figure}[t]
    \centering
\includegraphics[width=\textwidth,height=0.99\textheight,keepaspectratio]{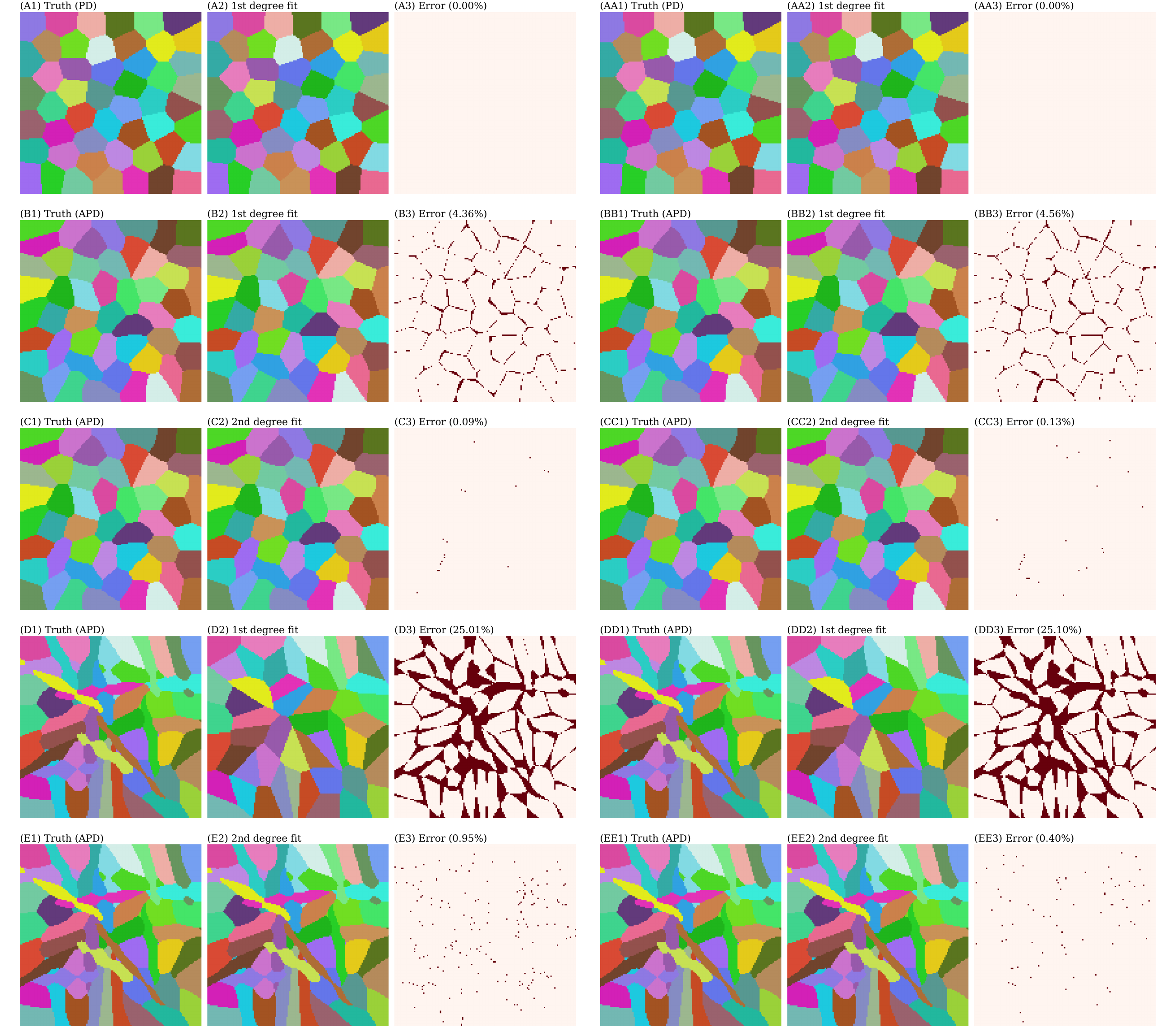}
    \caption{The reconstruction of (anisotropic) power diagrams as described in  Section~\ref{sec:recon-apd}. Column~1: Ground truth. The given grain map $G$ is generated by a power diagram (PD) or anisotropic power diagram (APD), where $G(x)=i$ if pixel $x$ is in cell $i$. The colours correspond to the assignment of pixels. Note that (B1)=(C1) is an APD with low anisotropy and (D1)=(E1) is an APD with high anisotropy. Column 2: A power diagram (1st degree fit) or anisotropic power diagram (2nd degree fit) is fitted to the grain map $G$ by approximately maximising $\Phi^G_{\e}$ with initial guess $\bm \theta^{(0)}=\mathbf{0}$. Column 3: The misassignment error. The light pink pixels are correctly assigned, the dark pixels are incorrectly assigned. Columns 4--6: The same as columns 1--3 but with the initial guess $\bm \theta^{(0)} = \bm \theta'$ from Section~\ref{sec:heuristic_guess}. In each case $\e = 0.01$, there are $N= 50$ cells and $|\Omega| = 19,600$ pixels, and the algorithm was run for 2000 iterations and took around 10-41 seconds.}
    \label{fig:1}
\end{figure}

\subsection{Fitting to EBSD grain maps}\label{sec:fit-ebsd}
To showcase the true strength of the proposed framework, we next fit Legendre polynomial diagrams to EBSD-derived grain maps. In contrast to the synthetic setting of Section~\ref{sec:recon-apd}, the EBSD grain boundaries are not expected to lie exactly in the algebraic class induced by a fixed degree $d$, so the problem only allows for imperfect reconstruction and a maximiser exists (Proposition~\ref{prop:existence-nonseparable}). We work with two EBSD datasets provided to us by our industrial partner, Tata Steel %Europe
\db{Netherlands}. 
The first, referred to as the \emph{small EBSD dataset}, yields a grain map $G\,\colon\,\Omega \to [N]$ with $N=245$ grains and $|\Omega| = 63,252$ pixels. The second, referred to as the \emph{big EBSD dataset}, results in a grain map with $N = 4686$ and $|\Omega| = 1,033,376$ pixels. The results are presented in Figure~\ref{fig:2} and Figure~\ref{fig:3}.

\begin{figure}[t]
    \centering
    \includegraphics[width=\textwidth,height=0.82\textheight,keepaspectratio]{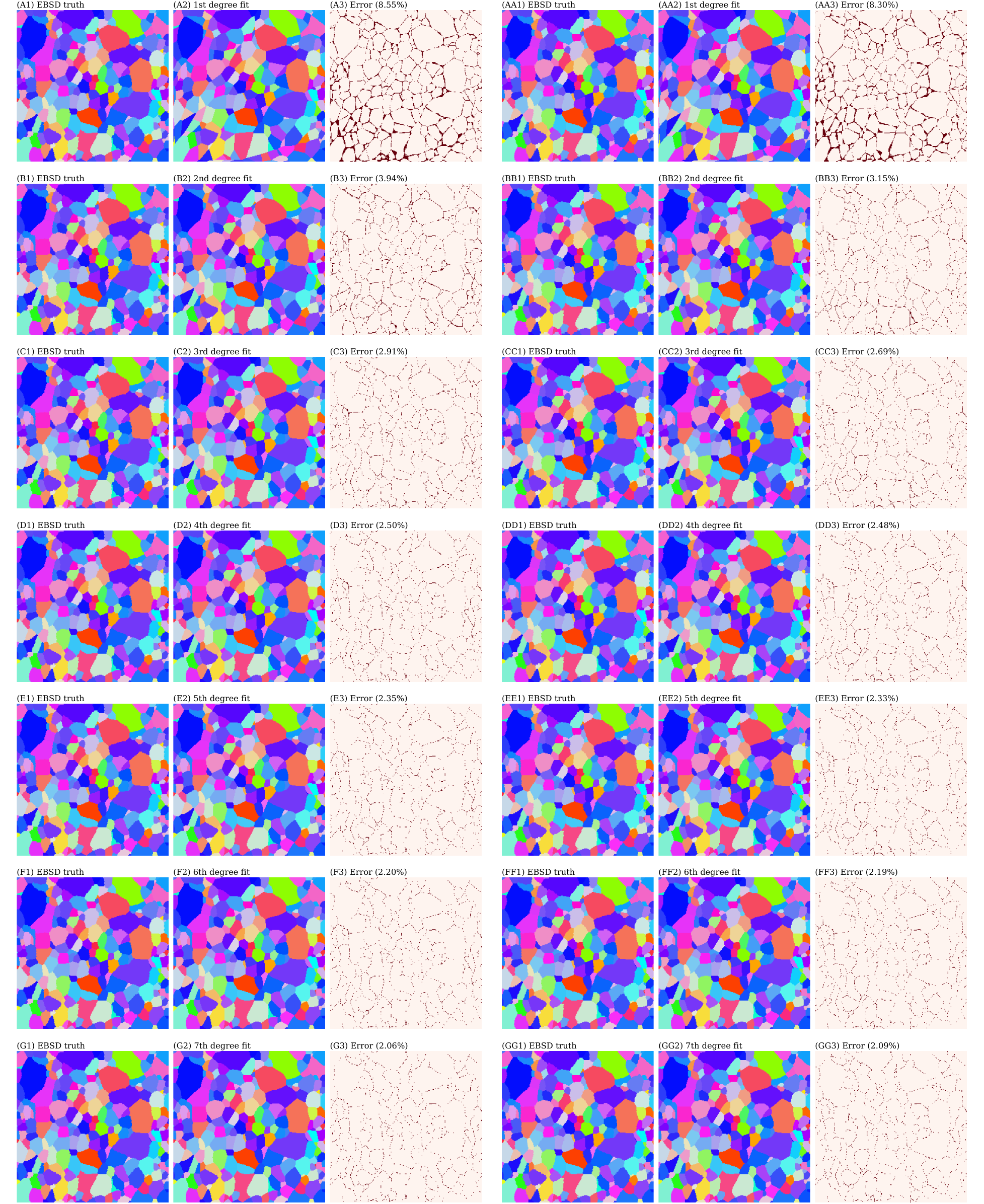}
    \caption{The fitting of polynomial diagrams to the small EBSD dataset as described in Section~\ref{sec:fit-ebsd}. 
    Column 1: Ground truth. The given grain map $G$ comes from an EBSD image of steel, where $G(x)=i$ if pixel $x$ is in grain $i$. The colours correspond to the assignment of pixels. Column 2, row $d$: A polynomial diagram of degree $d$, for $d \in [7]$, is fitted to the grain map $G$ by approximately maximising $\Phi^G_{\e}$ with initial guess $\bm \theta^{(0)}=\mathbf{0}$. Column 3: The misassignment error. The light pink pixels are correctly assigned, the dark pixels are incorrectly assigned. Columns 4--6: The same as columns 1--3 but with the initial guess $\bm \theta^{(0)} = \bm \theta'$ from Section~\ref{sec:heuristic_guess}. In each case $\e = 0.01$, there are $N= 245$ grains and $|\Omega| = 63,252$ pixels, and the algorithm was run for 1000 iterations and took around 25-64 seconds (with run time increasing with $d$).}
    \label{fig:2}
\end{figure}

\begin{figure}[t]
    \centering
    \includegraphics[width=\textwidth,height=0.82\textheight,keepaspectratio]{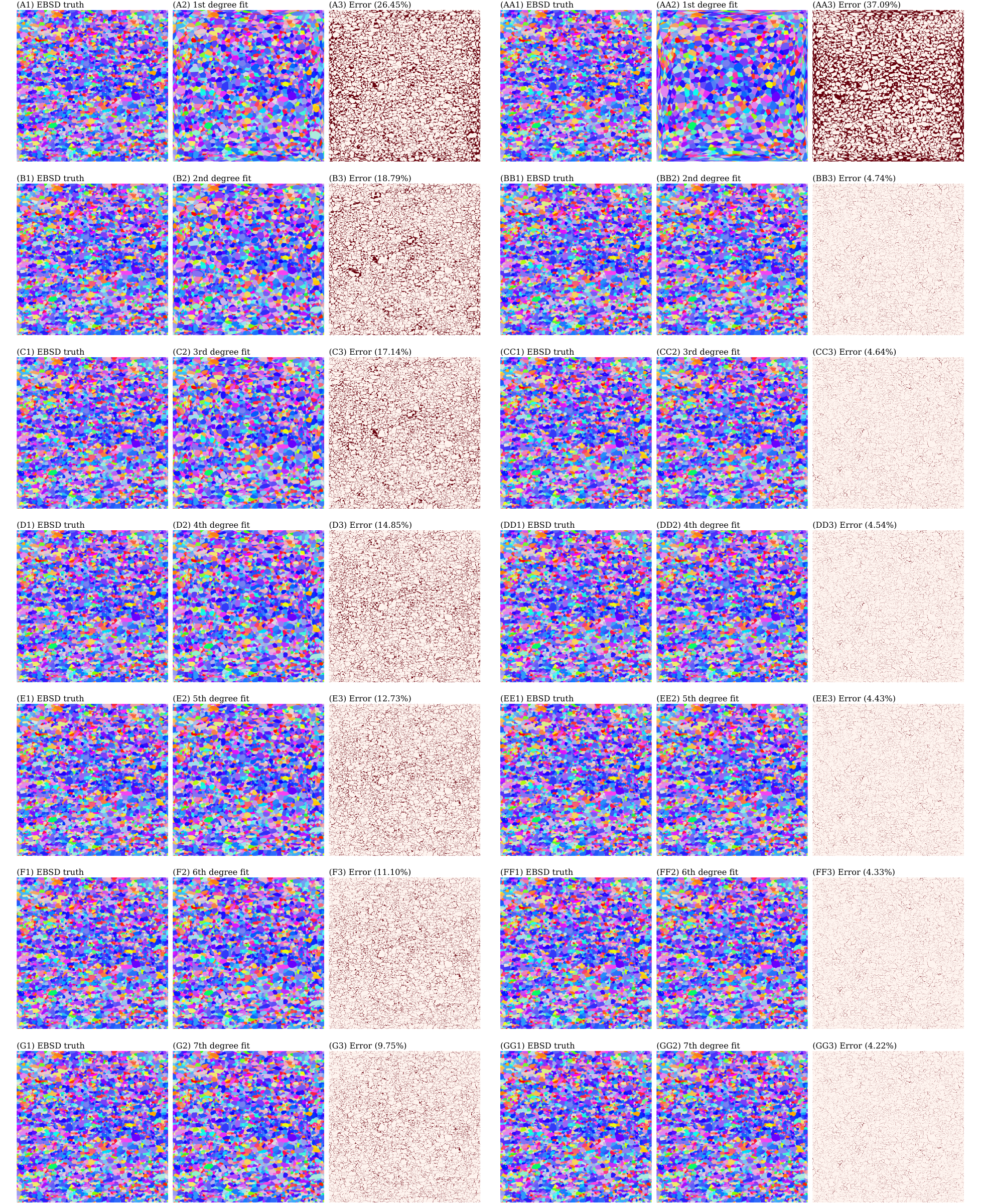}
    \caption{The fitting of polynomial diagrams to the big EBSD dataset as described in Section~\ref{sec:fit-ebsd}. 
    Column 1: Ground truth. The given grain map $G$ comes from an EBSD image of steel, where $G(x)=i$ if pixel $x$ is in grain $i$. The colours correspond to the assignment of pixels. Column 2, row $d$: A polynomial diagram of degree $d$, for $d \in [7]$, is fitted to the grain map $G$ by approximately maximising $\Phi^G_{\e}$ with initial guess $\bm \theta^{(0)}=\mathbf{0}$. Column 3: The misassignment error. The light pink pixels are correctly assigned, the dark pixels are incorrectly assigned. Columns 4--6: The same as columns 1--3 but with the initial guess $\bm \theta^{(0)} = \bm \theta'$ from Section~\ref{sec:heuristic_guess}. In each case $\e = 0.01$, there are $N= 4686$ grains and $|\Omega| = 1,033,376$ pixels, and the algorithm was run for 1000 iterations and took around 4.5-16 minutes (with run time increasing with $d$).
    }
    \label{fig:3}
\end{figure}

\subsection{Discussion}\label{sec:discussion}
We now report on several important aspects illuminated in our numerical experiments.
\subsubsection{Degree $d$.}\label{sec:degree_choice}
Increasing $d$ increases the number of parameters per grain to
\(
K_d=\frac{(d+1)(d+2)}{2}
\)
and permits more complex algebraic interfaces (Remark~\ref{rem:geom-boundaries}).
Empirically, the optimal value of the objective function $\Phi_\varepsilon^G$ improves monotonically with $d$, while the reconstruction accuracy on EBSD data improves up to a problem-dependent saturation degree. The improvement is mathematically obvious, since any lower-degree fit is an admissible higher-degree fit (with additional parameters set to zero).

A simple quantity that helps assess the trade-off between model complexity and data size is the \emph{compression ratio}
\begin{equation}\label{eqn:compression}
{\rm Compr}(d,N,|\Omega|)\;:=\;\frac{K_d\,N}{3|\Omega|}.
\end{equation}
Interpreting a labelled pixel set naively as storing, per pixel, two coordinates and one label,
a grain map requires $3|\Omega|$ scalar values, whereas a fitted degree-$d$ diagram is specified by $K_dN$ coefficients. \mbt{This convention is also reasonable in practice for EBSD data, where grain maps are often handled as unstructured pixel lists (for instance after removing small grains or leaving unindexed regions as gaps), so that the pixel locations are stored explicitly together with the grain labels}. We report on the compression ratio for
our EBSD experiments in Table~\ref{table:compression}.

\begin{table}[h!]
	\begin{center}
		\begin{tabular}{|l||c|c|c|c|c|c|c|c|}
			\hline
			\diagbox{data set}{Compr}{degree $d$} & 1 & 2 & 3 & 4 & 5 & 6 & 7 & 8 \\
			\hline
			Small EBSD & 0.39\% & 0.77\% & 1.29\% & 1.94\% & 2.71\% & 3.62\% & 4.65\% & 5.81\%  \\
			Big EBSD & 0.45\% & 0.91\% & 1.51\% & 2.27\% & 3.17\% & 4.23\% & 5.44\% & 6.80\%  \\
			\hline
		\end{tabular}
	\end{center}
	\caption{The compression ratio \eqref{eqn:compression} for the small and big EBSD datasets when reconstructed with a polynomial diagram of degree $d$ for $d \in [8]$.}
\label{table:compression}
\end{table}

\subsubsection{Moment-based heuristic initialisation for PDs and APDs.}\label{sec:heuristic_guess}
For the cases $d=1$ and $d=2$, an alternative to the fully ambiguous initialisation ${\bm \theta}^{(0)} = \bm 0$ is to use the well-known moment-based heuristic from~\cite{TR18}, which was derived for
the standard description of power diagrams in terms of $(\bm w,\bm y)$ and anisotropic power diagrams in terms of $(\bm w,\bm y, \bm{\mathrm{A}})$ (c.f.~Section~\ref{sec:rel-std-diagrams}), rather than in terms of $\bm \theta$. For a given grain map ${G:\Omega\to[N]}$, it can be computed as follows. First we compute first moments (centroids)
\begin{equation}\label{eq:centroid-guess}
	y'_i \;:=\; \frac{1}{|G_i|}\sum_{x\in G_i} x 
\end{equation}
and the (central) second-moment matrices
\begin{equation}\label{eq:cov-guess}
	(\mathrm{B}'_i)_{kl}
	\;:=\;
	\frac{1}{|G_i|}\sum_{x\in G_i} (x-y'_i)_k\,(x-y'_i)_l,
	\qquad k,l\in[2].
\end{equation}
\mbt{If $\mathrm{B}'_i$ is invertible, the anisotropy heuristic suggested in \cite{TR18} is}
\begin{equation}\label{eq:A-guess}
	\mathrm{A}'_i \;:=\; (\mathrm{B}'_i)^{-1},
\end{equation}
and the weight heuristic is
\begin{equation}
\label{eq:w-guess}    
w'_i = \frac{\sqrt{{\rm det}\mathrm{A}'_i}|G_i|}{|\Omega|\pi}
\end{equation}
(in the case of a PD, we take $\mathrm{A}_i' = {\rm Id}$ in \eqref{eq:w-guess}). Note that $w'_i$ is the ratio of the `area' of the grain $G_i$ divided by the area of the ellipse corresponding to $B_i'$. 
These heuristics give rise to an initial guess $({\bm w'},{\bm y'})$ for a PD and $({\bm w'},{\bm y'}, {\bm {\mathrm A}'})$ for an APD, which can then be converted to the linear
parameter representation ${\bm \theta'}$ via the transformations in Section~\ref{sec:rel-std-diagrams} and also to the Legendre basis as discussed in Remark~\ref{rem:basis_change}.

We advocate using this heuristic guess as an initialisation (even for $d \geq 3$ by setting higher degree coefficients to zero), provided that $\e$ is chosen small enough (see Section~\ref{sec:role-of-eps}). The benefit is clearly seen in Figure~\ref{fig:3},
where for $d \geq 2$ starting from the heuristic guess leads to a much reduced error in pixel assignment.
\subsubsection{Recovering a physically-relevant standard description.}
\label{sec: PD constraint}
When $d\in\{1,2\}$ we may optimise in the linear coefficient representation
$\bm\theta=(\theta_1,\dots,\theta_N)$ (with a gauge fix $\theta_N\equiv 0$), and then convert a posteriori
to the standard geometric parameters $({\bm w},{\bm y})$ (power diagrams) or $(\bm w,\bm y, \bm{\mathrm{A}})$ (anisotropic power diagrams) via equations \eqref{pd-lr-to-phys} and \eqref{apd-lr-to-phys}. One should keep in mind that this conversion is \emph{not canonical}: by Lemma~\ref{lem:gauge} the diagram depends only on the
equivalence class $\bm\theta \sim \bm\theta + (c,\dots,c)$, 
since adding the same $c\in\R^K$ shifts all scores by the same
function $x\mapsto c\cdot\eta(x)$ and hence does not change the argmin. Consequently, the recovered physical parameters are defined only up to such a common shift. Note that the diagram is also invariant under multiplicative rescaling, $\bm\theta \sim \lambda \bm \theta$, for any $\lambda > 0$, but this transformation does not leave the objective function $\Phi^G_{\e}$ invariant (unless $\e$ is also rescaled\mbt{, see Lemma~\ref{lem:eps-scaling}}). This is discussed further in Section \ref{sec:role-of-eps}.

Recall that for $d=2$ (APDs) the anisotropy matrices $\mathrm{A}_i$ are positive definite. However, after imposing the optimisation gauge $\theta_N=0$, the matrix $\mathrm{A}_N=\mathbf{0}$ is not positive definite, even though the induced diagram $\{L_i\}$ (defined in \eqref{eq:Li}) is perfectly meaningful. Similarly, maximising $\Phi^G_{\e}$ over all of $\mathbb{R}^{6 \times N}$ without gauge-fixing may lead to some $A_i$ that are not positive definite. However, the
gauge freedom can be used to enforce a physically-admissible anisotropy (i.e., $\mathrm{A}_i$ positive definite for all $i$) in post-processing.
Indeed, in the monomial basis \eqref{h-APD} the matrix $\mathrm{A}_i$
is read off from the quadratic coefficients as follows:
\[
\mathrm{A}_i
=\begin{pmatrix}\theta_{i,1}&\tfrac12\theta_{i,2}\\[1pt]\tfrac12\theta_{i,2}&\theta_{i,3}\end{pmatrix}.
\]
Adding to every $\theta_i$ a vector $c \in \mathbb{R}^6$ of the form $c=(c_1,c_2,c_3,0,0,0)$ amounts to replacing every
$A_i$ by $A_i+C$, where $C$ is the same symmetric matrix for all grains, while leaving the diagram $\{L_i\}$ (defined in \eqref{eq:Li}) unchanged. Choosing $C=\lambda I$ with 
\[
\lambda>-\min_{i \in [N]}\lambda_{\min}(\mathrm{A}_i),
\]
where $\lambda_{\min}(\mathrm{A}_i)$ denotes the smallest eigenvalue of $\mathrm{A}_i$, ensures that $\mathrm{A}_i+C$ is positive definite for all $i$. 
One then recomputes $(y_i,w_i)$ from \eqref{apd-lr-to-phys} using the shifted
coefficients. In summary, this procedure leaves the diagram $\{L_i\}$ unchanged while enforcing the constraint that the anisotropy matrices are positive definite, even though we solve the \emph{unconstrained} optimisation problem $\max \Phi^G_{\e}$. This observation was made in \cite{AFGK23}.
If optimisation is carried out in a different polynomial basis, e.g.\ Legendre, the same procedure applies after converting $\theta$ to the monomial basis. 

For $d=1$ (PDs), the same phenomenon appears in a simpler form: a common shift of the linear coefficients corresponds to a common translation of all recovered seeds $y_i$ (and a corresponding adjustment of the weights), again without changing the
diagram. This invariance is well-known in the optimal transport literature; see, e.g., \cite[Prop.~6]{Meyron2019}.

\subsubsection{Practical role of the regularisation parameter $\varepsilon$.}\label{sec:role-of-eps}
Lemma~\ref{lem:eps-scaling} shows that $\varepsilon$ can be absorbed into a rescaling of parameters:
\[
\Phi_\varepsilon^G(\bm\theta)=\Phi_1^G(\bm\theta/\varepsilon),
\]
and hence (assuming a maximiser exists) maximisers satisfy
\[
\bar{\bm\theta}\in\arg\max \Phi_1^G
\quad\Longrightarrow\quad
\varepsilon\,\bar{\bm\theta}\in\arg\max \Phi_\varepsilon^G.
\]
Separately, the induced hard assignment satisfies:
\[
F_{\varepsilon\bm\theta}(x)=F_{\bm\theta}(x)\qquad(\varepsilon>0),
\qquad\text{and hence}\qquad
{\rm Err}^G(\varepsilon\bm\theta)={\rm Err}^G(\bm\theta).
\]
Thus, for the exact problem in exact arithmetic, $\varepsilon$ carries no modelling content.

In computation, however, $\varepsilon$ has a pronounced effect on the optimisation trajectory. Empirically, the heuristic initialisation ${\bm \theta^{(0)}=\,\bm \theta'}$ from Section~\ref{sec:heuristic_guess} may yield a good hard reconstruction (small ${\rm Err}^G({\bm\theta'})$) while the corresponding values $h_{\theta'_i}(x)$ exhibit only weak separation between the winning index and its competitors on substantial subsets of pixels. With $\varepsilon\sim 1$ this typically leads to early L-BFGS iterates that substantially rearrange interfaces, and ${\rm Err}^G$ can increase markedly before decreasing again; in all the numerical experiments considered, this transient deterioration effectively removes the advantage of the heuristic relative to the fully ambiguous start $\bm\theta^{(0)}=\bm 0$.

By contrast, taking $\varepsilon\ll 1$ makes the initial iterate behave more deterministically and we observe a much more stable evolution of ${\rm Err}^G$, frequently with near-monotone decrease, albeit sometimes with slower progress in $\Phi_\varepsilon^G$ once the reconstruction is already good. 
Importantly, this behaviour can be reproduced while keeping $\varepsilon=1$ by rescaling the initial parameters: since $\Phi_\varepsilon^G(\bm\theta)=\Phi_1^G(\bm\theta/\varepsilon)$, running at small $\varepsilon$ from ${\bm\theta'}$ is equivalent (for the objective and its derivatives) to running at $\varepsilon=1$ from the enlarged initialisation ${\bm\theta'}/\varepsilon$.
Designing robust annealing strategies based on this scaling (and stopping criteria aligned with ${\rm Err}^G$) is an algorithmic issue and will be pursued in future work.

\section{Conclusions and outlook}\label{sec:conclusion}
We formulated a framework of parametrised minimisation diagrams (PMDs) and focused on the linear class $h_{\theta_i}(x)=\theta_i\cdot\eta(x)$, for which the logistic regression objective function $\Phi_\varepsilon^G$ is concave, and strictly concave after gauge fixing. Standard power diagrams and anisotropic power diagrams arise as low-degree instances of this framework, and polynomial diagrams provide a systematic extension in which cell interfaces are algebraic curves of controlled degree, which is the main novelty of this work in the context of microstructure modelling. 

On the computational side, the concave objective function admits reliable large-scale optimisation. In the perfect reconstruction regime (synthetic PD/APD ground truth) the supremum of $\Phi_\varepsilon^G$ is not attained, but nevertheless near-optimal approximate maximisers yield a perfect reconstruction (no misassigned pixels). In the imperfect reconstruction regime, typically arising when dealing with EBSD data, a maximiser exists, and increasing the polynomial degree improves accuracy but with diminishing returns. 

Several directions for further work appear natural. First, while degrees $d\ge 3$ empirically fit complex experimental boundaries well, their geometric and physical interpretation is less clear than for $d=1$ (PD) and $d=2$ (APD); understanding what classes of boundary regularity and anisotropy are efficiently represented by higher-degree 
diagrams is an open question. Second, the linear framework is not restricted to polynomial features: one may enrich the design function $\eta$ with additional basis functions (e.g., trigonometric or radial functions) to capture oscillatory or multiscale boundaries, while preserving the concavity of $\Phi_\varepsilon^G$.
Finally, incorporating a physically-motivated regularisation term in the objective function (e.g., penalising curvature \mbt{or encouraging mild shape constraints on the cells, such as approximate convexity or star-shapedness)} could improve robustness on noisy EBSD segmentations and facilitate the interpretation of fitted parameters.

\mbt{Another possible perspective is to view the present fitting problem as a semi-discrete inverse optimal transport problem. Indeed, in semi-discrete optimal transport, Laguerre cells arise from transporting 
%a
\db{an absolutely} 
continuous source measure to a discrete target measure under a parametrised cost. In our setting, one may heuristically interpret the labelled pixels as samples from such an optimal partition and the optimisation over PMD parameters as learning, within a restricted parametrised class, both the effective cost and the induced target masses from these samples. This is closely related in spirit to inverse optimal transport, where one seeks to infer the ground cost from samples of an observed coupling or matching
\cite{LiYeZhouZha2019,ChiuWangShafto2022,AndradePeyrePoon2024,Poon2025notes}. Making this connection precise, and clarifying what may be gained from the inverse-transport viewpoint both theoretically and computationally, would be an interesting direction for future work.}

\bibliographystyle{siamplain}
\bibliography{Bibliography}

\end{document}